\documentclass[11pt,a4paper]{amsart}

\usepackage[left=2.3cm, top=2.3cm,bottom=2.3cm,right=2.3cm]{geometry}

\usepackage{mathtools,amsmath,amssymb,amsthm,mathrsfs,calc,graphicx,stmaryrd,dsfont,tikz,pgfplots,bbm,float,array,epsfig,hyperref}
\usepackage[numbers,square]{natbib}
\usepackage[american]{babel}
\usepackage{amsfonts}              
\usepackage{enumitem}

\numberwithin{equation}{section}
\numberwithin{figure}{section}

\allowdisplaybreaks[4]

\newtheorem {theorem}{Theorem}[section]
\newtheorem {proposition}[theorem]{Proposition}
\newtheorem {lemma}[theorem]{Lemma}
\newtheorem {corollary}[theorem]{Corollary}

{\theoremstyle{definition}

\newtheorem*{convention*}{Convention}

\newtheorem {example}[theorem]{Example}
\newtheorem {remark}[theorem]{Remark}
}

\def\ba{\begin{array}}
\def\ea{\end{array}}
\def\bea{\begin{eqnarray} \label}
\def\eea{\end{eqnarray}}
\def\be{\begin{equation} \label}
\def\ee{\end{equation}}
\def\bit{\begin{itemize}}
\def\eit{\end{itemize}}
\def\ben{\begin{enumerate}}
\def\een{\end{enumerate}}

\def\P{\mathbb{P}}

\def\R{\mathbb{R}}
\def\RRd1{\mathbb{R}^{d+1}}

\def\dint{\textup{d}}

\newcommand{\eee}{{\rm e}}

\newcommand{\ind}{\mathbbm{1}}

\begin{document}

\title{Coverage probability of the planar Brownian convex hull}

\author{Zakhar Kabluchko}
\address{Zakhar Kabluchko: Institut f\"ur Mathematische Stochastik, Universit\"at M\"unster, Orl\'eans-Ring 10, 48149 M\"unster, Germany}
\email{zakhar.kabluchko@uni-muenster.de}

\date{}

\begin{abstract}
Let $\mathcal K$ be the convex hull of a planar Brownian motion up to time $1$. We explicitly determine the probability that $\mathcal K$ contains a prescribed point $z\in\mathbb R^2$. The proof is based on a general formula relating the radial function of a random convex body to its support function and its derivative, together with a comparison of $\mathcal K$ with an arcsine-rescaled Gaussian ellipse. We also obtain an analogous formula for the convex hull of a planar Brownian bridge. Finally, we derive small- and large-time asymptotics for the corresponding coverage probabilities.
\end{abstract}



\maketitle

\section{Introduction and main results}
\subsection{Introduction}
Let $(W_t)_{t\geq 0}$ be a standard Brownian motion in $\mathbb{R}^2$
starting at the origin, and let
\[
    \mathcal K:=\operatorname{conv}(W_t:0\leq t\leq 1)
\]
be its convex hull up to time $1$. The random convex set $\mathcal K$ is a
classical object in the study of planar Brownian motion. Almost surely, $\mathcal K$ is
compact, the origin belongs to its interior, and its boundary is a continuously
differentiable curve. The latter fact~\cite{Levy1948,CranstonHsuMarch1989}
is particularly striking in view of the roughness of Brownian paths; see also~\cite{AlexanderEldan2019} for its $d$-dimensional generalization.  The law
of $\mathcal K$ is invariant under rotations around the origin.

Several global geometric characteristics of $\mathcal K$ are known explicitly.
In particular, the expected perimeter and the expected area are given by
\[
    \mathbb E\,\operatorname{Per}(\mathcal K)=\sqrt{8\pi},
    \qquad
    \mathbb E\,\operatorname{Area}(\mathcal K)=\frac{\pi}{2},
\]
where the first formula goes back to Tak\'acs~\cite{Takacs1980}
(see also~\cite{BianeLetac2011} for an extension), and the second to
El Bachir~\cite{ElBachir1983}. Generalizations of these formulas to
$d$-dimensional Brownian motion were obtained in~\cite{Eldan2014} and
\cite{KabluchkoZaporozhets2016}.
For recent related results on convex and star hulls
of planar Brownian motion and Brownian bridges, see
\cite{Panzo2026,SandricSebekSimek2026,Sebek2024}.
Further related results concern scaling limits of random-walk convex hulls
\cite{WadeXu2015}, large deviations for the perimeter and area
\cite{AkopyanVysotsky2021}, convex hulls of L\'evy processes
\cite{MolchanovWespi2016}, high-dimensional absorption problems
\cite{TikhomirovYoussef2017}, and limit shapes for convex hulls of
stationary Gaussian processes \cite{DavydovDombry2012}.
Convex hulls of Brownian motion and related stochastic processes have also
been extensively studied in the statistical physics literature; see, for
example,
\cite{RandonFurlingMajumdarComtet2009,MajumdarMoriSchaweSchehr2021,
DeBruyneBenichouMajumdarSchehr2022}. We refer to \cite{MajumdarComtetRandonFurling2010} for a review of results
on random convex hulls.

\subsection{Main results}
Our purpose is to derive an explicit formula for the probability that a
prescribed point $z\in\R^2$ is covered by $\mathcal K$. In the terminology
of random-set theory~\cite[p.~33]{Molchanov}, the map \(z\mapsto \mathbb P(z\in\mathcal K)\)
is the \emph{coverage function} of $\mathcal K$. By rotation invariance, this probability depends
only on the Euclidean norm $\|z\|$.

\begin{theorem}[Coverage function of the planar Brownian convex hull]
\label{theo:coverage_function_brownian_hull}
For every point $z\in\R^2$,
\[
    \P(z\in\mathcal K)
    =
    \frac{4}{\pi}
    \int_0^{\pi/2}
    \overline{\Phi}\left(
        \frac{\|z\|}{\cos\theta}
    \right)\,\dint\theta
    =
    2\int_0^1
    \overline{\Phi}\left(\frac{\|z\|}{\sqrt t}\right)
    \frac{\dint t}{\pi\sqrt{t(1-t)}},
\]
where $\overline{\Phi}$ denotes the upper tail of the standard normal distribution:
\[
    \overline{\Phi}(x)
    =
    \frac{1}{\sqrt{2\pi}}
    \int_x^\infty \eee^{-s^2/2}\,\dint s.
\]
\end{theorem}

The preceding formula admits several probabilistic representations.
Let $\xi\sim N(0,1)$ be a standard normal random variable, let $\Theta$ be
uniformly distributed on $(0,\frac{\pi}{2})$, and suppose that $\xi$ and
$\Theta$ are independent. Also let $A:=\cos^2\Theta$. Then $A$ has the
arcsine distribution on $(0,1)$, that is,
\[
    \mathbb P(A\in\dint t)
    =
    \frac{\dint t}{\pi\sqrt{t(1-t)}},
    \qquad 0<t<1,
\]
and $A$ and $\xi$ are independent. Then Theorem~\ref{theo:coverage_function_brownian_hull} can equivalently be
expressed as
\[
    \mathbb P(z\in\mathcal K)
    =
    \mathbb P\bigl(|\xi|\cos\Theta>\|z\|\bigr)
    =
    \mathbb P\bigl(\sqrt A\,|\xi|>\|z\|\bigr).
\]
Further, let $(X_t)_{t\geq0}$ be a one-dimensional standard  Brownian motion independent of $A$.
Since $\max_{0\leq t\leq a}X_t$ has the same law as $\sqrt a\,|\xi|$ for
every $a>0$, we also have
\[
    \mathbb P(z\in\mathcal K)
    =
    \mathbb P\left(
        \max_{0\leq t\leq A}X_t>\|z\|
    \right).
\]

A closely related explicit formula holds for the convex hull of a planar Brownian bridge.
Let $(W_t^{\mathrm{br}})_{0\leq t\leq1}$ be a standard planar Brownian bridge
from $0$ to $0$, and define
\[
    \mathcal K^{\mathrm{br}}
    :=
    \operatorname{conv}(W_t^{\mathrm{br}}:0\leq t\leq1).
\]
With probability $1$, the set $\mathcal K^{\mathrm{br}}$ is compact and the origin belongs to its interior. Its law is invariant under rotations around the origin.
The expected perimeter and area of $\mathcal K^{\mathrm{br}}$ are known explicitly:
\[
    \mathbb E\,\operatorname{Per}(\mathcal K^{\mathrm{br}})
    =
    \sqrt{\frac{\pi^3}{2}},
    \qquad
    \mathbb E\,\operatorname{Area}(\mathcal K^{\mathrm{br}})
    =
    \frac{\pi}{3}.
\]
The perimeter formula is due to Goldman~\cite{Goldman1996}, while the area
formula was obtained in~\cite{MajumdarComtetRandonFurling2010}.
\begin{theorem}[Coverage function of the planar Brownian bridge convex hull]
\label{theo:coverage_function_brownian_bridge_hull}
For every point $z\in\mathbb R^2\setminus\{0\}$,
\[
    \mathbb P(z\in\mathcal K^{\mathrm{br}})
    =
    \int_0^1
    \exp\left\{
        -\frac{\|z\|^2}{2t(1-t)}
    \right\}
    \,\dint t
    =
    \|z\|^2\eee^{-\|z\|^2}
    \left(
        K_1(\|z\|^2)-K_0(\|z\|^2)
    \right),
\]
where $K_\lambda$ denotes the modified Bessel function of the second kind of order $\lambda\in\{0,1\}$:
\[
    K_\lambda(x)
    =
    \int_0^\infty
    \eee^{-x\cosh v}\cosh(\lambda v)\,\dint v,
    \qquad x>0.
\]
The integral representation in the first equality remains valid for
$z=0$,  with both sides equal to $1$.
\end{theorem}

Proofs of Theorems~\ref{theo:coverage_function_brownian_hull} and~\ref{theo:coverage_function_brownian_bridge_hull} will be given
in Section~\ref{sec:proof_main_theo_coverage_probab_brownian_hull}.
The proofs are intrinsically two-dimensional. A comparably explicit
extension of these results to Brownian convex hulls in dimensions $d\geq3$
would require different methods; this problem will be treated elsewhere.

\medskip
For convex hulls generated by finitely many independent random points in $\R^d$, the following results are known.  A classical
distribution-free formula for the probability that the origin belongs to the
convex hull of random points in $\R^d$, under central symmetry and a
general-position assumption, is due to Wendel~\cite{Wendel1962}. In dimension
two,  Jewell and Romano~\cite{JewellRomano1982} related the probability that
an arbitrary fixed point belongs to the convex hull of i.i.d.\ random
points to a circle-coverage problem; see also~\cite{JewellRomano1985} for a
more general study of inclusion probabilities for random convex hulls.
For Gaussian polytopes in arbitrary dimension, explicit formulas for the
coverage probability of an arbitrary fixed point were obtained in~\cite{KabluchkoZaporozhets2020}.

\subsection{Consequences and asymptotics} \label{subsec:consequences_asymptotics}
In the following we list some consequences of
Theorems~\ref{theo:coverage_function_brownian_hull} and~\ref{theo:coverage_function_brownian_bridge_hull}. Proofs will be given in
Section~\ref{sec:proofs_corollaries}.

The first result concerns the radial parts of the two coverage functions:
their derivatives admit especially simple closed forms. For a unit vector
$u\in\mathbb R^2$, define
\[
    p_{\mathcal K}(r)
    :=
    \mathbb P(ru\in\mathcal K),
    \qquad
    p_{\mathcal K^{\mathrm{br}}}(r)
    :=
    \mathbb P(ru\in\mathcal K^{\mathrm{br}}),
    \qquad r\geq0.
\]
By rotation invariance, these functions do not depend on the choice of $u$,
and clearly
\[
    p_{\mathcal K}(0)
    =
    p_{\mathcal K^{\mathrm{br}}}(0)
    =
    1.
\]

\begin{proposition}[Radial derivatives of the coverage functions]
\label{prop:radial_derivative_coverage_function}
For every $r>0$, we have
\[
    p_{\mathcal K}'(r)
    =
    -\frac{\sqrt{2}}{\pi^{3/2}}
    \eee^{-r^2/4}
    K_0\left(\frac{r^2}{4}\right),
    \qquad
    p_{\mathcal K^{\mathrm{br}}}'(r)
    =
    -2r\,\eee^{-r^2}
    K_0(r^2).
\]
Here $K_0$ denotes the modified Bessel function of the second kind of order $0$.
\end{proposition}

Next we mention two asymptotic results. Let $\mathcal K_T$ be the convex hull
of the planar Brownian motion up to time $T>0$, that is,
\[
    \mathcal K_T:=\operatorname{conv}(W_t:0\leq t\leq T).
\]
Furthermore, let $(W_t^{\mathrm{br},T})_{t\in [0,T]}$ be a standard
planar Brownian bridge from $0$ to $0$ on the time interval $[0,T]$, and define
\[
    \mathcal K_T^{\mathrm{br}}
    :=
    \operatorname{conv}(W_t^{\mathrm{br},T}:0\leq t\leq T).
\]
Then $\mathcal K=\mathcal K_1$ and
$\mathcal K^{\mathrm{br}}=\mathcal K_1^{\mathrm{br}}$, and Brownian
scaling gives
\[
    \mathcal K_T\stackrel{\mathrm d}=\sqrt T\,\mathcal K,
    \qquad
    \mathcal K_T^{\mathrm{br}}
    \stackrel{\mathrm d}=\sqrt T\,\mathcal K^{\mathrm{br}}.
\]

The two asymptotic results concern the probabilities of two rare events:
long-time non-coverage of a fixed point and very short-time coverage.
To motivate the first result, let $(X_t)_{t\geq0}$ be a standard
one-dimensional Brownian motion and fix $x\in\mathbb R\setminus\{0\}$.
By the reflection principle,
\begin{align*}
    \mathbb P\bigl(X_t\neq x \text{ for all } 0\leq t\leq T\bigr)
    &=
    1-2\overline{\Phi}\left(\frac{|x|}{\sqrt T}\right)
    \sim
    \sqrt{\frac{2}{\pi}}\,
    \frac{|x|}{\sqrt T},
    \\
    \mathbb P\bigl(
        X_t\neq x \text{ for all } 0\leq t\leq T
        \,\big|\, X_T=0
    \bigr)
    &=
    1-\exp\left\{-\frac{2x^2}{T}\right\}
    \sim
    \frac{2x^2}{T},
\end{align*}
as $T\to\infty$.
Probabilities that a stochastic process does not cross a prescribed level up
to a large time are commonly referred to as persistence probabilities; see,
for example, \cite{BrayMajumdarSchehr2013} for a review. The following
proposition gives two-dimensional analogues of these one-dimensional
persistence asymptotics, with point avoidance replaced by non-coverage by
the Brownian motion and Brownian bridge convex hulls.
\begin{proposition}[Long-time non-coverage probabilities]
\label{prop:asympt_large_T}
Fix any $z\in\mathbb R^2\setminus\{0\}$. Then, as $T\to\infty$,
\[
    \mathbb P(z\notin\mathcal K_T)
    \sim
    \frac{\sqrt{2}\,\|z\|}{\pi^{3/2}\sqrt T}\log T,
    \qquad
    \mathbb P(z\notin\mathcal K_T^{\mathrm{br}})
    \sim
    \frac{\|z\|^2}{T}\log T.
\]
\end{proposition}
A closely related discrete analogue was obtained by Vysotsky and Zaporozhets
\cite{VysotskyZaporozhets2018}. For a planar symmetric random walk
$S_1,S_2,\ldots$ whose increment distribution assigns zero mass to every
affine line, they proved
\[
    \mathbb P\bigl(0\notin\operatorname{conv}(S_1,\ldots,S_n)\bigr)
    \sim \frac{\log n}{\sqrt{\pi n}}.
\]
This asymptotic was subsequently extended to arbitrary fixed dimension
$d\geq2$ in \cite[Theorem~5.1]{KabluchkoVysotskyZaporozhets2017}.
Analogous absorption results for random walk bridges are also obtained in
\cite{VysotskyZaporozhets2018,KabluchkoVysotskyZaporozhets2017}.
Thus, in both the discrete and Brownian settings, the two-dimensional
non-coverage probability acquires an additional logarithmic factor compared
with the corresponding one-dimensional persistence asymptotic.

At the opposite end of the time scale, small-time coverage of a fixed nonzero point is exponentially unlikely
for both the Brownian-motion and Brownian-bridge convex hulls. More precisely, we have the following
asymptotics.

\begin{proposition}[Small-time coverage probabilities]
\label{prop:asympt_small_T}
Fix any $z\in\mathbb R^2\setminus\{0\}$. Then, as $T\downarrow0$,
\[
    \mathbb P(z\in\mathcal K_T)
    \sim
    \frac{2T}{\pi\|z\|^2}
    \eee^{-\|z\|^2/(2T)},
    \qquad
    \mathbb P(z\in\mathcal K_T^{\mathrm{br}})
    \sim
    \sqrt{\frac{\pi T}{8\|z\|^2}}\,
    \eee^{-2\|z\|^2/T}.
\]
\end{proposition}

As a final consequence of
Theorems~\ref{theo:coverage_function_brownian_hull}
and~\ref{theo:coverage_function_brownian_bridge_hull},
we obtain explicit formulas for the expected radial moments of
$\mathcal K_T$ and $\mathcal K_T^{\mathrm{br}}$.

\begin{proposition}[Expected radial moments of the Brownian and Brownian bridge hulls]
\label{prop:expected_radial_moments}
For every $T>0$ and $\alpha>-2$, we have
\begin{align*}
    \mathbb{E}\left[\int_{\mathcal K_T}\|x\|^\alpha\,\dint x\right]
    &=
    T^{1+\frac{\alpha}{2}}
    \frac{2^{2+\frac{\alpha}{2}}}{\alpha+2}
    \frac{
        \Gamma\left(\frac{\alpha+3}{2}\right)^2
    }{
        \Gamma\left(\frac{\alpha+4}{2}\right)
    },
    \\
    \mathbb{E}\left[\int_{\mathcal K_T^{\mathrm{br}}}\|x\|^\alpha\,\dint x\right]
    &=
    T^{1+\frac{\alpha}{2}}
    \frac{2^{2+\frac{\alpha}{2}}\pi}{\alpha+2}
    \frac{
        \Gamma\left(\frac{\alpha+4}{2}\right)^3
    }{
        \Gamma(\alpha+4)
    }.
\end{align*}
\end{proposition}

\begin{example}
Taking $\alpha=0$ and $\alpha=2$ in
Proposition~\ref{prop:expected_radial_moments}, we obtain
\[
    \mathbb E\operatorname{Area}(\mathcal K_T)
    =
    \frac{\pi T}{2},
    \quad
    \mathbb E\operatorname{Area}(\mathcal K_T^{\mathrm{br}})
    =
    \frac{\pi T}{3},
\quad
    \mathbb E \, I(\mathcal K_T)
    =
    \frac{9\pi T^2}{16},
    \quad
    \mathbb E \, I(\mathcal K_T^{\mathrm{br}})
    =
    \frac{2\pi T^2}{15},
\]
where
$
    I(K):=\int_K \|x\|^2\,\dint x
$
denotes the moment of inertia of a convex body $K$ about the origin.
\end{example}

\section{Coverage functions of Brownian motion and Brownian bridge convex hulls}
\label{sec:proof_main_theo_coverage_probab_brownian_hull}

\subsection{Introduction}

In this section we prove
Theorems~\ref{theo:coverage_function_brownian_hull}
and~\ref{theo:coverage_function_brownian_bridge_hull}.
Let $\mathcal Q$ be a random compact convex set in $\mathbb R^2$ such that
$0\in\operatorname{int}\mathcal Q$ a.s.
For $\theta\in\mathbb R$, let
\[
    n_\theta:=(\cos\theta,\sin\theta)\in\mathbb R^2.
\]
The \emph{support function} of $\mathcal Q$ is defined by
\[
    h_{\mathcal Q}(\theta)
    :=
    \max_{z\in\mathcal Q}\langle z,n_\theta\rangle,
    \qquad \theta\in\mathbb R.
\]
The \emph{radial function} of $\mathcal Q$ is defined by
\[
    \rho_{\mathcal Q}(\theta)
    :=
    \max\{r\geq0:r n_\theta\in\mathcal Q\},
    \qquad \theta\in\mathbb R.
\]

To prove Theorem~\ref{theo:coverage_function_brownian_hull}, it suffices
to identify the distribution of the radial function
$\rho_{\mathcal K}(\theta)$ of the Brownian convex hull $\mathcal K$.
By rotational invariance, this distribution does not depend on $\theta$.
Thus, for every $z\in\mathbb R^2$,
\[
    \mathbb P(z\in\mathcal K)
    =
    \mathbb P\bigl(\rho_{\mathcal K}(0)\geq\|z\|\bigr).
\]
The proof proceeds in four steps.
\begin{itemize}
    \item We first derive a general identity relating the distribution of
    $\rho_{\mathcal Q}(\theta)$ to the support function
    $h_{\mathcal Q}$ and its derivative. The formula is valid for a general
    random convex body $\mathcal Q$.

    \item The resulting identity depends, at each angle $\theta$, only on the
    joint distribution of the ``sufficient statistic''
    \[
        \bigl(h_{\mathcal Q}(\theta),h_{\mathcal Q}'(\theta)\bigr).
    \]
    Consequently, two random convex bodies whose sufficient statistics have
    the same distribution for almost every $\theta$ have radial functions
    with the same one-dimensional distributions.

    \item We then introduce an arcsine-rescaled Gaussian ellipse
    $\mathcal E$ and show that the pairs
    \[
        \bigl(h_{\mathcal K}(\theta),h_{\mathcal K}'(\theta)\bigr)
        \quad\text{and}\quad
        \bigl(h_{\mathcal E}(\theta),h_{\mathcal E}'(\theta)\bigr)
    \]
    have the same distribution for every fixed $\theta$. It follows that
    $\rho_{\mathcal K}(\theta)$ and $\rho_{\mathcal E}(\theta)$ have the
    same distribution.

    \item Finally, we compute the distribution of
    $\rho_{\mathcal E}(\theta)$ explicitly. This yields a distributional
    representation of $\rho_{\mathcal K}(\theta)$ and, in turn, the
    coverage formula of Theorem~\ref{theo:coverage_function_brownian_hull}.
\end{itemize}

The proof of Theorem~\ref{theo:coverage_function_brownian_bridge_hull}
follows the same strategy and is given in the final subsection.
For convenience, the statements of several standard analytic tools used below,
including the change-of-variables formula and Danskin's theorem, are collected
in Appendix~\ref{sec:standard_facts}.

\subsection{A formula for the radial function}

We begin with a general identity expressing the distribution of the radial function in terms of the support function and its derivative.

\begin{lemma}[Identity for the radial function]\label{lem:coarea_radial_function}
Let $\mathcal Q$ be a random compact convex set in $\mathbb R^2$ such that
$0\in\operatorname{int}\mathcal Q$ a.s.
Then, for every $\mu\in\mathbb R$ and every nonnegative Borel function
$f:(0,\infty)\to[0,\infty]$,
\begin{align}
    &\int_0^\infty
    f(r)\,
    \mathbb P\bigl(\rho_{\mathcal Q}(\mu)<r\bigr)
    \,\dint r
    \notag\\
    &\qquad=
    \frac12
    \int_{-\pi/2}^{\pi/2}
    \mathbb E\left[
        f\left(
            \frac{h_{\mathcal Q}(\mu+\theta)}{\cos\theta}
        \right)
        \frac{
            \left|
                h_{\mathcal Q}'(\mu+\theta)\cos\theta
                +
                h_{\mathcal Q}(\mu+\theta)\sin\theta
            \right|
        }{\cos^2\theta}
    \right]
    \,\dint\theta.
    \label{eq:general_coarea_radial_function}
\end{align}
Both sides are allowed to take the value $+\infty$.
\end{lemma}

\begin{remark}
For every realization of $\mathcal Q$, the function
$\theta\mapsto h_{\mathcal Q}(\theta)$ is Lipschitz~\cite[Theorem~H.1]{Molchanov}
and hence absolutely continuous. Thus its derivative
$h_{\mathcal Q}'(\theta)$ is well defined for Lebesgue-a.e.\ $\theta$. At points at which the derivative does not exist, we set
$h_{\mathcal Q}'(\theta):=0$.
\end{remark}

\begin{proof}[Proof of Lemma~\ref{lem:coarea_radial_function}]
We first consider $\mu=0$. For $\theta\in\mathbb R$, write
\[
    n_\theta=(\cos\theta,\sin\theta),
    \qquad
    n_\theta^\perp=(-\sin\theta,\cos\theta).
\]

The supporting lines of $\mathcal Q$ are indexed by
$\theta\in(-\pi,\pi]$: for every such $\theta$ there is a unique
supporting line of $\mathcal Q$ with outer normal vector $n_\theta$,
and its parametric representation is
\[
    t\mapsto h_{\mathcal Q}(\theta)n_\theta+t n_\theta^\perp,
    \qquad t\in\mathbb R.
\]
For $\theta\in(-\pi/2,\pi/2)$, this line intersects the positive
$x$-axis at the point $(R_{\mathcal Q}(\theta),0)$, where
\begin{equation}\label{eq:R_mathcal_Q_def}
    R_{\mathcal Q}(\theta)
    :=
    \frac{h_{\mathcal Q}(\theta)}{\cos\theta},
    \qquad
    -\frac{\pi}{2}<\theta<\frac{\pi}{2}.
\end{equation}
For $\theta\in(-\pi,\pi]\setminus(-\pi/2,\pi/2)$, the supporting line
does not intersect the positive $x$-axis. For $r>0$, let
$N_{\mathcal Q}(r)$ be the number of supporting lines of $\mathcal Q$
passing through $(r,0)$. Thus
\begin{equation}\label{eq:N_mathcal_Q_as_crossing_number}
    N_{\mathcal Q}(r)
    =
    \#\left\{
        \theta\in(-\pi/2,\pi/2):
        R_{\mathcal Q}(\theta)=r
    \right\}.
\end{equation}

On the other hand, if $0\in\operatorname{int}\mathcal Q$, then
elementary convex geometry implies that $N_{\mathcal Q}(r)=2$ if
$(r,0)\notin\mathcal Q$ and $N_{\mathcal Q}(r)=0$ if $(r,0)$ belongs
to the interior of $\mathcal Q$; see Lemma~\ref{lem:tangent_lines}.
Thus
\[
    N_{\mathcal Q}(r)
    =
    2\,\ind_{\{\rho_{\mathcal Q}(0)<r\}}
\]
for every $r>0$ except possibly $r=\rho_{\mathcal Q}(0)$. It follows
that, for every nonnegative Borel function
$f:(0,\infty)\to[0,\infty]$,
\[
    \int_0^\infty
    f(r)\,
    \ind_{\{\rho_{\mathcal Q}(0)<r\}}
    \,\dint r
    =
    \frac12\int_0^\infty
    f(r)\,
    N_{\mathcal Q}(r)
    \,\dint r.
\]

Let us evaluate the right-hand side using
\eqref{eq:N_mathcal_Q_as_crossing_number}. Since $\mathcal Q$ is compact, its support function is Lipschitz. By
\eqref{eq:R_mathcal_Q_def}, the function $R_{\mathcal Q}$ is locally
Lipschitz on $(-\pi/2,\pi/2)$, and hence differentiable for
Lebesgue-a.e.\ $\theta$.  Moreover, for Lebesgue-a.e.\
$\theta\in(-\pi/2,\pi/2)$,
\[
    R_{\mathcal Q}'(\theta)
    =
    \frac{
        h_{\mathcal Q}'(\theta)\cos\theta
        +
        h_{\mathcal Q}(\theta)\sin\theta
    }{\cos^2\theta}.
\]
The change-of-variables formula, see
Lemma~\ref{lem:one_dimensional_coarea_formula}, gives
\[
    \int_0^\infty
    f(r)\,
    \ind_{\{\rho_{\mathcal Q}(0)<r\}}
    \,\dint r
    =
    \frac12\int_0^\infty
    f(r)\,
    N_{\mathcal Q}(r)
    \,\dint r
    =
    \frac12\int_{-\pi/2}^{\pi/2}
    f(R_{\mathcal Q}(\theta))
    |R_{\mathcal Q}'(\theta)|
    \,\dint\theta
\]
for almost every realization of $\mathcal Q$. Taking expectations and
using Tonelli's theorem proves
\eqref{eq:general_coarea_radial_function} for $\mu=0$.

For general $\mu$, let $\widetilde{\mathcal Q}$ be the rotation of
$\mathcal Q$ through the angle $-\mu$. Then $\rho_{\widetilde{\mathcal Q}}(0)=\rho_{\mathcal Q}(\mu)$,
whereas
\[
    h_{\widetilde{\mathcal Q}}(\theta)
    =
    h_{\mathcal Q}(\mu+\theta),
    \qquad
    h_{\widetilde{\mathcal Q}}'(\theta)
    =
    h_{\mathcal Q}'(\mu+\theta)
\]
for Lebesgue-a.e.\ $\theta$. Applying the already proved case $\mu=0$
to $\widetilde{\mathcal Q}$ yields
\eqref{eq:general_coarea_radial_function}.
\end{proof}
\begin{corollary}[Comparison by the support function and its derivative]\label{cor:radial_comparison_support_statistics}
Let $\mathcal Q_1$ and $\mathcal Q_2$ be random compact convex sets in
$\mathbb R^2$ such that $0\in\operatorname{int}\mathcal Q_i$ almost surely,  $i=1,2$.
Suppose that, for Lebesgue-a.e.\ $\theta\in\mathbb R$,
\[
    \bigl(
        h_{\mathcal Q_1}(\theta),
        h_{\mathcal Q_1}'(\theta)
    \bigr)
    \stackrel{\mathrm d}=
    \bigl(
        h_{\mathcal Q_2}(\theta),
        h_{\mathcal Q_2}'(\theta)
    \bigr).
\]
Then, for every $\mu\in\mathbb R$,
\[
    \rho_{\mathcal Q_1}(\mu)
    \stackrel{\mathrm d}=
    \rho_{\mathcal Q_2}(\mu).
\]
\end{corollary}

\begin{proof}
Fix $\mu\in\mathbb R$. By Lemma~\ref{lem:coarea_radial_function} and the assumed equality in
distribution, for every nonnegative Borel function
$f:(0,\infty)\to[0,\infty]$,
\[
    \int_0^\infty
    f(r)\,
    \mathbb P\bigl(\rho_{\mathcal Q_1}(\mu)<r\bigr)
    \,\dint r
    =
    \int_0^\infty
    f(r)\,
    \mathbb P\bigl(\rho_{\mathcal Q_2}(\mu)<r\bigr)
    \,\dint r.
\]
It follows that
\[
    \mathbb P\bigl(\rho_{\mathcal Q_1}(\mu)<r\bigr)
    =
    \mathbb P\bigl(\rho_{\mathcal Q_2}(\mu)<r\bigr)
\]
for Lebesgue-a.e.\ $r>0$. Both sides are nondecreasing and
left-continuous functions of $r$, hence they are equal for every
$r>0$. This proves the claim.
\end{proof}

\subsection{The Brownian hull and a Gaussian ellipse}
Now we apply the preceding results to  two random convex bodies. The first one is the convex hull of the standard planar Brownian motion up to time $1$, denoted by
\[
    \mathcal K
    =
    \operatorname{conv}(W_t:0\leq t\leq 1).
\]
The second random convex body is the arcsine-rescaled Gaussian ellipse $\mathcal E$ defined as follows.  Let $A$ have the arcsine distribution on $(0,1)$, and let $G_0,G_1$ be independent standard Gaussian vectors in $\mathbb R^2$,
independent of $A$. Let $L:\mathbb R^2\to\mathbb R^2$ be the random
linear map defined by
\[
    Le_1=G_0,
    \qquad
    Le_2=G_1,
\]
where $e_1= (1,0)$ and $e_2 = (0,1)$ form the standard orthonormal basis of $\R^2$. Then the \emph{arcsine-rescaled Gaussian ellipse} is defined as
\[
    \mathcal E
    :=
    \sqrt A\,L^{\mathsf T}\overline{\mathbb B}^2,
\]
where $\overline{\mathbb B}^2$ is the closed unit disk in $\R^2$.
Denote the support functions of $\mathcal K$ and $\mathcal E$ by
$h_{\mathcal K}$ and $h_{\mathcal E}$, respectively, and their radial
functions by $\rho_{\mathcal K}$ and $\rho_{\mathcal E}$.

\begin{lemma}[Equality in distribution of sufficient statistics]\label{lem:support_statistics_K_E}
For every fixed $\theta\in\mathbb R$,
\[
    \bigl(
        h_{\mathcal K}(\theta),
        h_{\mathcal K}'(\theta)
    \bigr)
    \stackrel{\mathrm d}=
    \bigl(
        h_{\mathcal E}(\theta),
        h_{\mathcal E}'(\theta)
    \bigr).
\]
More precisely, their common distribution is the distribution of
\[
    \bigl(\sqrt A\,R,\sqrt A\,\xi\bigr),
\]
where $A$, $R$, and $\xi$ are independent, $A$ has the arcsine
distribution on $(0,1)$, $\xi\sim N(0,1)$, and $R$ has the Rayleigh (or $\chi_2$)
density $r\mapsto r\eee^{-r^2/2}$, $r>0$.
\end{lemma}

\begin{proof}
Recall that for  $\theta\in\mathbb R$  we write
\[
    n_\theta=(\cos\theta,\sin\theta),
    \qquad
    n_\theta^\perp=(-\sin\theta,\cos\theta).
\]

\medskip
\noindent
\emph{Part 1.} We first consider $\mathcal K$.  Write $W_t=(X_t,Y_t)$, where $(X_t)_{t\geq 0}$ and $(Y_t)_{t\geq 0}$ are independent standard one-dimensional Brownian
motions.
For fixed $\theta\in \mathbb R$, set
\[
    B_t^{(\theta)}
    :=
    \langle W_t,n_\theta\rangle = X_t\cos\theta+Y_t\sin\theta,
    \qquad
    C_t^{(\theta)}
    :=
    \langle W_t,n_\theta^\perp\rangle = -X_t\sin\theta+Y_t\cos\theta, \qquad t\geq 0.
\]
The processes $(B^{(\theta)}_t)_{t\geq 0}$ and $(C^{(\theta)}_t)_{t\geq 0}$ are independent standard
one-dimensional Brownian motions. Define also
\[
    \tau_\theta
    :=
    \operatorname*{argmax}_{0\leq t\leq 1}B_t^{(\theta)}.
\]
For every $\theta$, the maximum is attained at a unique time almost surely. Note that $h_{\mathcal K}(\theta) = B^{(\theta)}_{\tau_\theta} = X_{\tau_\theta}\cos\theta+Y_{\tau_\theta}\sin\theta$. Danskin's formula, see Lemma~\ref{lem:danskin_special_case_cos_sin},  yields: For every $\theta$, there is a null set in the probability space on which $(W_t)_{t\geq 0}$ is defined such that outside this null set,
\begin{equation*}
    h_{\mathcal K}'(\theta)
    =
    -X_{\tau_\theta}\sin \theta+Y_{\tau_\theta}\cos\theta
    =
    C_{\tau_\theta}^{(\theta)}.
\end{equation*}

We now characterize the joint distribution of $(h_{\mathcal K}(\theta), h_{\mathcal K}'(\theta))$ or, equivalently, $(B^{(\theta)}_{\tau_\theta}, C_{\tau_\theta}^{(\theta)})$.  The random variable $\tau_\theta$ has the arcsine distribution on
$(0,1)$. Moreover, conditionally on $\tau_\theta=t$,
\[
    B_{\tau_\theta}^{(\theta)}
    \stackrel{\mathrm d}=
    \sqrt t\,R,
\]
where $R$ has the Rayleigh (or $\chi_2$) density $r\mapsto r\eee^{-r^2/2}$, $r>0$; see~\cite[p.~425]{Shepp1979} or~\cite[Section~2.8]{KaratzasShreve1991}.
Since $(C^{(\theta)}_t)_{t\geq 0}$ is independent of $(B^{(\theta)}_t)_{t\geq0}$, conditionally on
$\tau_\theta=t$ the random variable
$C_{\tau_\theta}^{(\theta)}=C_t^{(\theta)}$ is independent of
$B_{\tau_\theta}^{(\theta)}$ and has distribution $N(0,t)$. Hence
\[
    \mathcal L\left(
        h_{\mathcal K}(\theta),
        h_{\mathcal K}'(\theta)
        \,\middle|\,
        \tau_\theta=t
    \right)
    =
    \mathcal L\left(
        \sqrt t\,R,
        \sqrt t\,\xi
    \right),
\]
where $R$ has the Rayleigh distribution,  $\xi\sim N(0,1)$, and $R$ and $\xi$ are independent. Since $\tau_\theta$ has
the same arcsine distribution as $A$, it follows that
\[
    \bigl(
        h_{\mathcal K}(\theta),
        h_{\mathcal K}'(\theta)
    \bigr)
    \stackrel{\mathrm d}=
    \bigl(
        \sqrt A\,R,
        \sqrt A\,\xi
    \bigr).
\]

\medskip
\noindent
\emph{Part 2.}
We next consider $\mathcal E$. Recall that
\[
    \mathcal E=\sqrt A\,L^{\mathsf T}\overline{\mathbb B}^2,
    \qquad
    Le_1=G_0,\quad Le_2=G_1,
\]
where $G_0,G_1$ are independent standard Gaussian vectors in
$\mathbb R^2$, independent of $A$. Its support function is
\[
    h_{\mathcal E}(\theta)
    =
    \sqrt A\,
    \|G_0\cos\theta+G_1\sin\theta\|.
\]
Put
\[
    U_\theta
    :=
    G_0\cos\theta+G_1\sin\theta,
    \qquad
    V_\theta
    :=
    -G_0\sin\theta+G_1\cos\theta.
\]
By rotational invariance of the Gaussian distribution,
$U_\theta$ and $V_\theta$ are independent standard Gaussian vectors
in $\mathbb R^2$. Since $U_\theta\neq0$ almost surely,
\[
    h_{\mathcal E}'(\theta)
    =
    \sqrt A\,
    \left\langle
        \frac{U_\theta}{\|U_\theta\|},
        V_\theta
    \right\rangle.
\]
Now $\|U_\theta\|$ has the Rayleigh density
$r\mapsto r\eee^{-r^2/2}$, while, conditionally on $U_\theta$,
\[
    \left\langle
        \frac{U_\theta}{\|U_\theta\|},
        V_\theta
    \right\rangle
    \sim N(0,1).
\]
Since this conditional distribution does not depend on the value of $U_\theta$,
the two random variables are independent. Thus
\[
    \bigl(
        h_{\mathcal E}(\theta),
        h_{\mathcal E}'(\theta)
    \bigr)
    \stackrel{\mathrm d}=
    \bigl(
        \sqrt A\,R,
        \sqrt A\,\xi
    \bigr),
\]
where $A,R,\xi$ are independent and have the distributions stated in
the lemma. This proves the claim.
\end{proof}

\subsection{Distributional representation for the radial functions}
By the preceding results, $\rho_{\mathcal K}(\theta)$ has the same distribution as $\rho_{\mathcal E}(\theta)$, for every $\theta$. Thus it remains to identify the distribution of $\rho_{\mathcal E}(\theta)$, which is elementary.
\begin{lemma}[Radial function of the Gaussian ellipse]\label{lem:radial_function_gaussian_ellipse}
For every fixed $\theta\in\mathbb R$,
\[
    \rho_{\mathcal E}(\theta)
    \stackrel{\mathrm d}=
    \sqrt A\,|\xi|,
\]
where $A$ has the arcsine distribution on $(0,1)$, $\xi\sim N(0,1)$,
and $A$ and $\xi$ are independent. In particular, for every $r\geq0$,
\[
    \mathbb P\bigl(\rho_{\mathcal E}(\theta)\geq r\bigr)
    =
    \int_0^1
    \mathbb P\left(
        |\xi|\geq\frac{r}{\sqrt t}
    \right)
    \frac{\dint t}{\pi\sqrt{t(1-t)}}.
\]
\end{lemma}

\begin{proof}
By rotational invariance, it is enough to consider $\theta=0$. Recall that
\[
    \mathcal E
    =
    \sqrt A\,L^{\mathsf T}\overline{\mathbb B}^2,
    \qquad
    Le_1=G_0,\quad Le_2=G_1,
\]
where $G_0,G_1$ are independent standard Gaussian vectors in
$\mathbb R^2$, independent of $A$. Thus
\[
    \mathcal E
    =
    \left\{
        \sqrt A\,
        \bigl(
            \langle G_0,y\rangle,
            \langle G_1,y\rangle
        \bigr)
        :
        y\in \overline{\mathbb B}^2
    \right\}.
\]

Let $V$ be a unit vector orthogonal to $G_1$. Then a point $(r,0)$,
where $r\geq0$, belongs to $\mathcal E$ if and only if there exists
$y\in \overline{\mathbb B}^2$ such that
\[
    \sqrt A\,\langle G_0,y\rangle=r,
    \qquad
    \langle G_1,y\rangle=0.
\]
Since the vectors in $\overline{\mathbb B}^2$ orthogonal to $G_1$ are precisely
$sV$, $|s|\leq1$, it follows that
\[
    \rho_{\mathcal E}(0)
    =
    \sqrt A\,|\langle G_0,V\rangle|.
\]

Conditionally on $G_1$, the unit vector $V$ is fixed, while $G_0$ is an
independent standard Gaussian vector. Hence $\langle G_0,V\rangle\sim N(0,1)$,
and its conditional distribution does not depend on $G_1$. Moreover,
$A$ is independent of $G_0$ and $G_1$. Consequently,
\[
    \rho_{\mathcal E}(0)
    \stackrel{\mathrm d}=
    \sqrt A\,|\xi|,
\]
where $\xi\sim N(0,1)$ is independent of $A$.

Finally, the distribution of $\mathcal E$ is rotationally invariant,
so the same identity holds for every $\theta\in\mathbb R$. The formula
for the tail probability follows by conditioning on $A$.
\end{proof}

Combining the preceding results yields the following distributional representation of the radial function of $\mathcal K$, together with a sequence of equivalent representations involving the support-function derivatives of $\mathcal K$ and $\mathcal E$.

\begin{proposition}[Distributional identities for the radial function and support function derivative]\label{prop:radial_support_derivative_distribution}
For every fixed $\theta\in\mathbb R$,
\[
    |h_{\mathcal K}'(\theta)|
    \stackrel{\mathrm d}=
    |h_{\mathcal E}'(\theta)|
    \stackrel{\mathrm d}=
    \sqrt A\,|\xi|
    \stackrel{\mathrm d}=
    \rho_{\mathcal E}(\theta)
    \stackrel{\mathrm d}=
    \rho_{\mathcal K}(\theta),
\]
where $A$ has the arcsine distribution on $(0,1)$,
$\xi\sim N(0,1)$, and $A$ and $\xi$ are independent.
\end{proposition}

\begin{proof}
The first two distributional equalities follow from
Lemma~\ref{lem:support_statistics_K_E}. The third follows from
Lemma~\ref{lem:radial_function_gaussian_ellipse}, and the last one
from Corollary~\ref{cor:radial_comparison_support_statistics} combined with Lemma~\ref{lem:support_statistics_K_E}.
\end{proof}

\subsection{Proof of Theorem~\ref{theo:coverage_function_brownian_hull}}
We are now ready to complete the proof of Theorem~\ref{theo:coverage_function_brownian_hull}. Let $z\in\mathbb R^2$. By the definition of the radial function, rotational invariance and  Proposition~\ref{prop:radial_support_derivative_distribution},
\[
    \mathbb P(z\in\mathcal K)
    =
    \mathbb P(\rho_{\mathcal K} (0) \geq  \|z\|)
    =
    \mathbb P(\sqrt A\,|\xi|\geq \|z\|).
\]
Conditioning on $A$ gives
\[
    \mathbb P(z\in\mathcal K)
    =
    2\int_0^1
    \overline{\Phi}\left(\frac{\|z\|}{\sqrt t}\right)
    \frac{\dint t}{\pi\sqrt{t(1-t)}}.
\]
Finally, the substitution $t=\cos^2\theta$ yields
\[
    \mathbb P(z\in\mathcal K)
    =
    \frac{4}{\pi}
    \int_0^{\pi/2}
    \overline{\Phi}\left(\frac{\|z\|}{\cos\theta}\right)
    \,\dint\theta,
\]
which proves the theorem.

\begin{remark}[Another random set with the same coverage function as $\mathcal K$]\label{rem:another_set_same_radial}
We have seen that $\mathcal K$ and the arcsine-rescaled Gaussian ellipse
$\mathcal E$ have the same coverage function. Let us mention another random set with this property, this time a
non-convex one. Let $U$ be uniformly distributed on the unit circle
$\mathbb S^1\subset\mathbb R^2$, and let $\xi\sim N(0,1)$ be independent
of $U$. Let $\mathcal C$ be the union of the two closed disks having the segments
$[0,|\xi|U]$ and $[0,-|\xi|U]$ as diameters. Then
\[
    \mathbb P(z\in\mathcal C)
    =
    \mathbb P(z\in\mathcal K),
    \qquad z\in\mathbb R^2.
\]
Indeed, a direct computation gives $\rho_{\mathcal C}(\theta):= \max\{r\geq0:r n_\theta\in\mathcal C\} =|\xi|\,|\langle U,n_\theta\rangle|$.
Since $|\langle U,n_\theta\rangle|^2$ has the arcsine distribution on
$(0,1)$ and is independent of $\xi$,
\[
    \rho_{\mathcal C}(\theta)
    \stackrel{\mathrm d}=
    \sqrt A\,|\xi|
    \stackrel{\mathrm d}=
    \rho_{\mathcal K}(\theta),
\]
by Proposition~\ref{prop:radial_support_derivative_distribution}.
The asserted equality of the coverage functions follows.
\end{remark}

\subsection{Proof of Theorem~\ref{theo:coverage_function_brownian_bridge_hull}}

Write
$W_t^{\mathrm{br}}
    =
    (X_t^{\mathrm{br}},Y_t^{\mathrm{br}})$, $0\leq t\leq1$,
where $(X^{\mathrm{br}}_t)_{t\in [0,1]}$ and $(Y^{\mathrm{br}}_t)_{t\in [0,1]}$ are independent standard
one-dimensional Brownian bridges. We shall compare
$\mathcal K^{\mathrm{br}}$ with a suitable Gaussian ellipse, similarly to the
Brownian motion case.

Let $V$ be uniformly distributed on $(0,1)$, and let $G_0,G_1$ be
independent standard Gaussian vectors in $\mathbb R^3$, independent of $V$.
Let $L:\mathbb R^2\to\mathbb R^3$ be a random linear map defined by
\[
    Le_1=G_0,
    \qquad
    Le_2=G_1,
\]
where $e_1= (1,0)$ and $e_2 = (0,1)$. Let $\overline{\mathbb B}^3$ be the closed unit ball in $\R^3$, and set
\[
    \mathcal E^{\mathrm{br}}
    :=
    \sqrt{V(1-V)}\,L^{\mathsf T}\overline{\mathbb B}^3.
\]

\begin{lemma}[Equality in distribution of sufficient statistics]
\label{lem:support_statistics_bridge_gaussian_ellipse}
For every fixed $\theta\in\mathbb R$,
\[
    \bigl(
        h_{\mathcal K^{\mathrm{br}}}(\theta),
        h_{\mathcal K^{\mathrm{br}}}'(\theta)
    \bigr)
    \stackrel{\mathrm d}=
    \bigl(
        h_{\mathcal E^{\mathrm{br}}}(\theta),
        h_{\mathcal E^{\mathrm{br}}}'(\theta)
    \bigr).
\]
More precisely, their common distribution is the distribution of
\[
    \sqrt{V(1-V)}\,(R_3,\xi),
\]
where $V$, $R_3$, and $\xi$ are independent,
$\xi\sim N(0,1)$, and $R_3$ has the $\chi_3$ distribution.
\end{lemma}

\begin{proof}
For fixed $\theta\in \R$, put
\[
    B_t^{(\theta)}
    :=
    X_t^{\mathrm{br}}\cos\theta+Y_t^{\mathrm{br}}\sin\theta,
    \qquad
    C_t^{(\theta)}
    :=
    -X_t^{\mathrm{br}}\sin\theta+Y_t^{\mathrm{br}}\cos\theta,
    \qquad t\in [0,1],
\]
and let $\tau_\theta$ be the a.s.\ unique time at which $(B^{(\theta)}_t)_{t\in [0,1]}$ attains its
maximum. Then $(B^{(\theta)}_t)_{t\in [0,1]}$ and $(C^{(\theta)}_t)_{t\in [0,1]}$ are independent standard
Brownian bridges and
\[
    h_{\mathcal K^{\mathrm{br}}}(\theta)
    =
    B_{\tau_\theta}^{(\theta)},
    \qquad
    h_{\mathcal K^{\mathrm{br}}}'(\theta)
    =
    C_{\tau_\theta}^{(\theta)}.
\]
The second identity follows from Lemma~\ref{lem:danskin_special_case_cos_sin}. The classical joint law of the maximum and its location for a Brownian
bridge implies that $\tau_\theta$ is uniform on $(0,1)$ and, conditionally
on $\tau_\theta=t$,
\[
    B_{\tau_\theta}^{(\theta)}
    \stackrel{\mathrm d}=
    \sqrt{t(1-t)}\,R_3,
\]
where $R_3$ has the $\chi_3$ distribution.  This result follows from the joint density of the maximum, its location,
and the endpoint of Brownian motion given in~\cite{Shepp1979}, after
conditioning on the endpoint being zero. Moreover,
\[
    C_{\tau_\theta}^{(\theta)}
    \mid\{\tau_\theta=t\}
    \sim N(0,t(1-t)),
\]
and the processes $(B^{(\theta)}_t)_{t\in [0,1]}$ and $(C^{(\theta)}_t)_{t\in [0,1]}$ are independent. Hence
\[
    \bigl(
        h_{\mathcal K^{\mathrm{br}}}(\theta),
        h_{\mathcal K^{\mathrm{br}}}'(\theta)
    \bigr)
    \stackrel{\mathrm d}=
    \sqrt{V(1-V)}\,(R_3,\xi).
\]
On the other hand,
\[
    h_{\mathcal E^{\mathrm{br}}}(\theta)
    =
    \sqrt{V(1-V)}
    \|G_0\cos\theta+G_1\sin\theta\|.
\]
The same argument as in Part~2 of  the proof of
Lemma~\ref{lem:support_statistics_K_E} shows that
\[
    \bigl(
        h_{\mathcal E^{\mathrm{br}}}(\theta),
        h_{\mathcal E^{\mathrm{br}}}'(\theta)
    \bigr)
    \stackrel{\mathrm d}=
    \sqrt{V(1-V)}\,(R_3,\xi),
\]
which proves the claim.
\end{proof}

By Corollary~\ref{cor:radial_comparison_support_statistics}, $\rho_{\mathcal K^{\mathrm{br}}}(\theta)$ has the same distribution as  $\rho_{\mathcal E^{\mathrm{br}}}(\theta)$, for every  $\theta\in\mathbb R$. It remains to identify the law of $\rho_{\mathcal E^{\mathrm{br}}}(\theta)$.

\begin{lemma}[Radial function of the Gaussian ellipse]
\label{lem:radial_function_bridge_gaussian_ellipse}
For every fixed $\theta\in\mathbb R$,
\[
    \rho_{\mathcal E^{\mathrm{br}}}(\theta)
    \stackrel{\mathrm d}=
    \sqrt{V(1-V)}\,R_2,
\]
where $V\sim\operatorname{Unif}(0,1)$ and $R_2$ is an independent
$\chi_2$ random variable.
\end{lemma}

\begin{proof}
By rotational invariance, it suffices to take $\theta=0$. As in the proof
of Lemma~\ref{lem:radial_function_gaussian_ellipse},
\[
    \rho_{\mathcal E^{\mathrm{br}}}(0)
    =
    \sqrt{V(1-V)}
    \left\|
        \operatorname{proj}_{G_1^\perp}G_0
    \right\|.
\]
Conditionally on $G_1$, the space $G_1^\perp\subset\mathbb R^3$ is
two-dimensional, and the projection of the independent standard Gaussian
vector $G_0$ onto this space is a standard two-dimensional Gaussian vector.
Its norm therefore has the $\chi_2$ distribution.
\end{proof}

\begin{proof}[Completion of the proof of
Theorem~\ref{theo:coverage_function_brownian_bridge_hull}]
Let $r=\|z\|$. By Lemma~\ref{lem:radial_function_bridge_gaussian_ellipse} and the comparison above,
\[
\begin{aligned}
    \mathbb P(z\in\mathcal K^{\mathrm{br}})
    =
    \mathbb P(\rho_{\mathcal E^{\mathrm{br}}}(0) \geq r)
    =
    \mathbb P\left(
        \sqrt{V(1-V)}\,R_2\geq r
    \right)
=
    \int_0^1
    \exp\left\{
        -\frac{r^2}{2t(1-t)}
    \right\}
    \,\dint t,
\end{aligned}
\]
since $\mathbb P(R_2\geq x)=\eee^{-x^2/2}$. For the representation in terms of $K_0$ and $K_1$, set
\begin{equation}\label{eq:F_def_bridge_function}
    F(x)
    :=
    \int_0^1
    \exp\left\{
        -\frac{x}{2t(1-t)}
    \right\}
    \,\dint t,
    \qquad x\geq 0.
\end{equation}
For $x>0$, differentiating under the integral sign and then using the substitution
$v=\log(t/(1-t))$, we obtain
\[
    F'(x)
    =
    -\frac12
    \int_0^1
    \frac{1}{t(1-t)}
    \exp\left\{
        -\frac{x}{2t(1-t)}
    \right\}
    \,\dint t
    =
    -\frac12
    \int_{-\infty}^{\infty}
    \exp\{-x(1+\cosh v)\}
    \,\dint v
    =
    -\eee^{-x}K_0(x).
\]
Using $K_0'(x)=-K_1(x)$ and  $K_1'(x)=-K_0(x)-\frac{1}{x}K_1(x)$, see \cite[Eq.~10.29.2 and Eq.~10.29.3]{DLMF}, we obtain
\[
    \frac{\dint}{\dint x}
    \left[
        x\eee^{-x}\bigl(K_1(x)-K_0(x)\bigr)
    \right]
    =
    -\eee^{-x}K_0(x) = F'(x).
\]
For each fixed $\lambda$, the standard large-argument asymptotic
\[
    K_\lambda(x)
    \sim
    \sqrt{\frac{\pi}{2x}}\,\eee^{-x},
    \qquad x\to\infty,
\]
holds; see \cite[Eq.~10.40.2]{DLMF}. Hence $x\eee^{-x}(K_1(x)-K_0(x)) \to 0$ as $x\to\infty$. By dominated convergence, $F(x) \to 0$ as $x\to\infty$. Since both functions tend to $0$ as $x\to\infty$ and have the same derivative,  they must coincide. Hence
\[
    F(x)
    =
    x\eee^{-x}\bigl(K_1(x)-K_0(x)\bigr), \qquad x>0.
\]
Taking $x=\|z\|^2$ completes the proof.
\end{proof}

\begin{remark}
The preceding proof also yields simple distributional representations for
the support function, its derivative, and the radial function, both for $\mathcal K^{\mathrm{br}}$ and $\mathcal E^{\mathrm{br}}$. For every
fixed $\theta\in\mathbb R$,
\[
    h_{\mathcal K^{\mathrm{br}}}(\theta)
    \stackrel{\mathrm d}=
    h_{\mathcal E^{\mathrm{br}}}(\theta)
    \stackrel{\mathrm d}=
    \sqrt{V(1-V)}\,R_3,
\]
\[
    |h_{\mathcal K^{\mathrm{br}}}'(\theta)|
    \stackrel{\mathrm d}=
    |h_{\mathcal E^{\mathrm{br}}}'(\theta)|
    \stackrel{\mathrm d}=
    \sqrt{V(1-V)}\,R_1,
\]
\[
    \rho_{\mathcal K^{\mathrm{br}}}(\theta)
    \stackrel{\mathrm d}=
    \rho_{\mathcal E^{\mathrm{br}}}(\theta)
    \stackrel{\mathrm d}=
    \sqrt{V(1-V)}\,R_2,
\]
where $V\sim\operatorname{Unif}(0,1)$ and $R_k$ denotes a $\chi_k$
random variable independent of $V$. Thus the three quantities correspond,
after the same random scaling, to the $\chi_3$, $\chi_1$, and $\chi_2$
distributions, respectively. In particular, unlike in the Brownian motion
case, $|h_{\mathcal K^{\mathrm{br}}}'(\theta)|$ and
$\rho_{\mathcal K^{\mathrm{br}}}(\theta)$ do not have the same distribution.
\end{remark}

\section{Proofs of the consequences and asymptotics}\label{sec:proofs_corollaries}

In this section we prove the results stated in
Section~\ref{subsec:consequences_asymptotics}, in the order in which they
appear. Recall that, for any unit vector $u\in\mathbb R^2$,
\[
    p_{\mathcal K}(r)
    :=
    \mathbb P(ru\in\mathcal K),
    \qquad
    p_{\mathcal K^{\mathrm{br}}}(r)
    :=
    \mathbb P(ru\in\mathcal K^{\mathrm{br}}),
    \qquad r\geq0.
\]
By Theorems~\ref{theo:coverage_function_brownian_hull}
and~\ref{theo:coverage_function_brownian_bridge_hull},
\[
    p_{\mathcal K}(r)
    =
    \frac{4}{\pi}
    \int_0^{\pi/2}
    \overline{\Phi}\left(\frac{r}{\cos\theta}\right)
    \,\dint\theta,
    \qquad
    p_{\mathcal K^{\mathrm{br}}}(r)
    =
    \int_0^1
    \exp\left\{
        -\frac{r^2}{2t(1-t)}
    \right\}
    \,\dint t,
    \qquad r\geq0.
\]

\subsection{Proof of Proposition~\ref{prop:radial_derivative_coverage_function}}
We consider $\mathcal K$ first. Differentiation under the integral sign gives
\[
\begin{aligned}
    p_{\mathcal K}'(r)
    &=
    -\frac{4}{\pi\sqrt{2\pi}}
    \int_0^{\pi/2}
    \frac{1}{\cos\theta}
    \exp\left\{-\frac{r^2}{2\cos^2\theta}\right\}
    \,\dint\theta
\\
    &=
    -\frac{4}{\pi\sqrt{2\pi}}\,
    \eee^{-r^2/2}
    \int_0^\infty
    \frac{\eee^{-r^2x^2/2}}{\sqrt{1+x^2}}
    \,\dint x
\\
    &=
    -\frac{2}{\pi\sqrt{2\pi}}\,
    \eee^{-r^2/4}
    \int_0^\infty
    \eee^{-(r^2/4)\cosh v}
    \,\dint v
    \\
    &=
    -\frac{\sqrt{2}}{\pi^{3/2}}
    \eee^{-r^2/4}
    K_0\left(\frac{r^2}{4}\right),
\end{aligned}
\]
where we used successively $x=\tan\theta$ and then
$x=\sinh(v/2)$. The last step follows from the integral representation $K_0(x)=\int_0^\infty \eee^{-x\cosh v}\,\dint v$, $x>0$.

For the Brownian bridge, Theorem~\ref{theo:coverage_function_brownian_bridge_hull}
gives $p_{\mathcal K^{\mathrm{br}}}(r) =F(r^2)$, where $F(x)$ is the function defined in~\eqref{eq:F_def_bridge_function}.
As shown in the proof of Theorem~\ref{theo:coverage_function_brownian_bridge_hull},
$F'(x)=-\eee^{-x}K_0(x)$.
Hence, by the chain rule,
\[
    p_{\mathcal K^{\mathrm{br}}}'(r)
    =
    2rF'(r^2)
    =
    -2r\,\eee^{-r^2}K_0(r^2),
    \qquad r>0.
\]

\subsection{Proof of Proposition~\ref{prop:asympt_large_T}}
By Brownian scaling,
\[
    \mathbb P(z\notin\mathcal K_T)
    =
    1-p_{\mathcal K}\left(\frac{\|z\|}{\sqrt T}\right),
    \qquad
    \mathbb P(z\notin\mathcal K_T^{\mathrm{br}})
    =
    1-p_{\mathcal K^{\mathrm{br}}}\left(\frac{\|z\|}{\sqrt T}\right).
\]
As $T\to\infty$, it remains to determine the asymptotics of
$p_{\mathcal K}(r)$ and $p_{\mathcal K^{\mathrm{br}}}(r)$ as $r\downarrow 0$. As $x\downarrow0$, we have $K_0(x)=\log\frac1x+O(1)$, see Eqs.~(10.31.2) and~(10.25.2) in~\cite{DLMF}. Hence
Proposition~\ref{prop:radial_derivative_coverage_function} yields, as
$r\downarrow0$,
\[
    -p_{\mathcal K}'(r)
    =
    \frac{2\sqrt2}{\pi^{3/2}}\log\frac1r+O(1),
    \qquad
    -p_{\mathcal K^{\mathrm{br}}}'(r)
    =
    4r\log\frac1r+O(r).
\]
Since $p_{\mathcal K}(0)=p_{\mathcal K^{\mathrm{br}}}(0)=1$, integration gives
\[
    1-p_{\mathcal K}(r)
    =
    \frac{2\sqrt2}{\pi^{3/2}}\,r\log\frac1r+O(r),
    \qquad
    1-p_{\mathcal K^{\mathrm{br}}}(r)
    =
    2r^2\log\frac1r+O(r^2),
\]
as $r\downarrow 0$. Taking $r=\|z\|/\sqrt T$ proves
\[
    \mathbb P(z\notin\mathcal K_T)
    \sim
    \frac{\sqrt2\,\|z\|}{\pi^{3/2}\sqrt T}\log T,
    \qquad
    \mathbb P(z\notin\mathcal K_T^{\mathrm{br}})
    \sim
    \frac{\|z\|^2}{T}\log T, \qquad T\to\infty.
\]

\subsection{Proof of Proposition~\ref{prop:asympt_small_T}}
By Brownian scaling,
\[
    \mathbb P(z\in\mathcal K_T)
    =
    p_{\mathcal K}\left(\frac{\|z\|}{\sqrt T}\right),
    \qquad
    \mathbb P(z\in\mathcal K_T^{\mathrm{br}})
    =
    p_{\mathcal K^{\mathrm{br}}}\left(\frac{\|z\|}{\sqrt T}\right).
\]
As $T\downarrow0$, it remains to determine the large-$r$ asymptotics of
$p_{\mathcal K}(r)$ and $p_{\mathcal K^{\mathrm{br}}}(r)$. Using $K_0(x)\sim \sqrt{\frac {\pi}{2x}}\,\eee^{-x}$ as  $x\to\infty$, see Eq.~(10.25.3) or Eq.~(10.40.2) in~\cite{DLMF}, Proposition~\ref{prop:radial_derivative_coverage_function} yields
\begin{align*}
    -p_{\mathcal K}'(r)
    &=
    \frac{\sqrt{2}}{\pi^{3/2}}
    \eee^{-r^2/4}
    K_0\left(\frac{r^2}{4}\right)
    \sim
    \frac{2}{\pi r}\,\eee^{-r^2/2},
    \\
    -p_{\mathcal K^{\mathrm{br}}}'(r)
    &=
    2r\,\eee^{-r^2}K_0(r^2) \sim \sqrt{2\pi}\,\eee^{-2r^2},
\end{align*}
as $r\to\infty$.  Since $p_{\mathcal K}(r)\to0$ and $p_{\mathcal K^{\mathrm{br}}}(r)\to0$ as $r\to\infty$, we obtain
\begin{align*}
    p_{\mathcal K}(r)
    &=
    \int_r^\infty -p_{\mathcal K}'(s)\,\dint s
    \sim
    \frac{2}{\pi}
    \int_r^\infty \frac1s\,\eee^{-s^2/2}\,\dint s
    \sim
    \frac{2}{\pi r^2}\,\eee^{-r^2/2},
    \\
        p_{\mathcal K^{\mathrm{br}}}(r)
    &=
    \int_r^\infty -p_{\mathcal K^{\mathrm{br}}}'(s)\,\dint s
    \sim
    \sqrt{2\pi}\int_r^\infty \eee^{-2s^2}\,\dint s
    \sim
    \sqrt{\frac{\pi}{8r^2}}\,\eee^{-2r^2},
\end{align*}
as $r\to\infty$.
Taking $r=\|z\|/\sqrt T$ gives the claimed formulas
\[
    \mathbb P(z\in\mathcal K_T)
    \sim
    \frac{2T}{\pi\|z\|^2}
    \eee^{-\|z\|^2/(2T)},
    \qquad
    \mathbb P(z\in\mathcal K_T^{\mathrm{br}})
    \sim
    \sqrt{\frac{\pi T}{8\|z\|^2}}\,
    \eee^{-2\|z\|^2/T}, \qquad T\downarrow 0.
\]

\subsection{Proof of Proposition~\ref{prop:expected_radial_moments}}
Let $\mathcal Q$ denote either $\mathcal K_T$ or $\mathcal K_T^{\mathrm{br}}$.
By polar coordinates and rotational invariance,
\begin{equation}\label{eq:E_int_x_alpha}
    \mathbb E\left[\int_{\mathcal Q}\|x\|^\alpha\,\dint x\right]
    =
    \frac{2\pi}{\alpha+2}\,
    \mathbb E\left[\rho_{\mathcal Q}(0)^{\alpha+2}\right].
\end{equation}
Consider $\mathcal K_T$ first. By Brownian scaling and
Proposition~\ref{prop:radial_support_derivative_distribution},
$\rho_{\mathcal K_T}(0)$ has the same law as $\sqrt{TA}\,|\xi|$,
where $A$ and $\xi$ are independent,
$A\sim\operatorname{Beta}(1/2,1/2)$ and $\xi\sim N(0,1)$. Hence
\[
    \mathbb E\bigl[\rho_{\mathcal K_T}(0)^{\alpha+2}\bigr]
    =
    T^{(\alpha+2)/2}
    \frac{\Gamma\left(\frac{\alpha+3}{2}\right)}
         {\sqrt\pi\,\Gamma\left(\frac{\alpha+4}{2}\right)}
    \,
    \frac{2^{(\alpha+2)/2}
          \Gamma\left(\frac{\alpha+3}{2}\right)}
         {\sqrt\pi}.
\]
Substitution into~\eqref{eq:E_int_x_alpha} gives the claimed formula for $\mathcal K_T$.
For the Brownian bridge, the preceding results and Brownian scaling give
\[
    \rho_{\mathcal K_T^{\mathrm{br}}}(0)
    \stackrel{\mathrm d}=
    \sqrt{TV(1-V)}\,R_2,
\]
where $V\sim\operatorname{Unif}(0,1)$ and $R_2\sim\chi_2$ are independent.
Therefore
\begin{align*}
    \mathbb E\bigl[\rho_{\mathcal K_T^{\mathrm{br}}}(0)^{\alpha+2}\bigr]
    &=
    T^{(\alpha+2)/2}
    \mathbb E\bigl[(V(1-V))^{(\alpha+2)/2}\bigr]
    \mathbb E R_2^{\alpha+2}
    \\
    &=
    T^{(\alpha+2)/2}
    \frac{\Gamma\left(\frac{\alpha+4}{2}\right)^2}
         {\Gamma(\alpha+4)}
    \,
    2^{(\alpha+2)/2}
    \Gamma\left(\frac{\alpha+4}{2}\right).
\end{align*}
Substitution into~\eqref{eq:E_int_x_alpha} gives the claimed formula for
$\mathcal K_T^{\mathrm{br}}$.

\appendix
\section{Auxiliary results}\label{sec:standard_facts}
In this appendix, we collect several standard results used in the proofs above, including the change-of-variables formula and Danskin's theorem.

\begin{lemma}[Tangent lines to a convex body]\label{lem:tangent_lines}
Let $K\subset\mathbb R^2$ be a compact convex set with nonempty interior.
\begin{enumerate}
    \item If $z\notin K$, then there are exactly two supporting lines of $K$ passing through $z$.
    \item If $z\in\operatorname{int} K$, then there is no supporting line of $K$ passing through $z$.
\end{enumerate}
\end{lemma}
\begin{proof}
If $z\notin K$, then the cone $\operatorname{pos}(K-z)$ is a proper closed convex cone with nonempty interior, whose boundary in $\mathbb R^2$ consists of exactly two rays. The two lines through $z$ generated by these boundary rays are precisely the supporting lines of $K$ passing through $z$. If $z\in\operatorname{int}K$, every line through $z$ has points of $K$ on both sides and therefore cannot be a supporting line.
\end{proof}

The next result is a one-dimensional case of the area formula (or change-of-variables
formula); see, for example, \cite[Theorem~3.9, p.~122]{EvansGariepy2015}.
The locally Lipschitz case stated below follows by exhaustion by compact subintervals.

\begin{lemma}[Change-of-variables formula]\label{lem:one_dimensional_coarea_formula}
Let $I\subset\mathbb R$ be an interval and let $g:I\to\mathbb R$
be locally Lipschitz. (So $g'$ exists Lebesgue-a.e.) Then, for every nonnegative Borel function
$f:\mathbb R\to[0,\infty]$,
\[
    \int_I f(g(x))\,|g'(x)|\,\dint x
    =
    \int_{\mathbb R} f(y)\,N_g(y)\,\dint y,
\]
where
\[
    N_g(y)
    :=
    \#\{x\in I:g(x)=y\}.
\]
Both sides are allowed to take the value $+\infty$.
The interval $I$ may be open, closed, half-open, bounded or unbounded.
\end{lemma}

The next result is a standard consequence of Danskin's theorem; see, for example, Theorem~4.13 and Remark~4.14 in~\cite{BonnansShapiro2000} for a more general version. For completeness, we provide a proof.
\begin{lemma}[Danskin's theorem, unique-maximizer version]
\label{lem:danskin}
Let $I\subset\mathbb R$ be an open interval, let $K$ be a compact
metric space, and let $F:I\times K\to\mathbb R$ be a continuous function. Define
\[
    H(\theta):=\max_{t\in K}F(\theta,t),
    \qquad \theta\in I.
\]
Suppose that the partial derivative
$\partial_\theta F(\theta,t)$ exists for every $(\theta,t)\in I\times K$
and is continuous on $I\times K$.  Fix $\theta_0\in I$ and suppose that $t\mapsto F(\theta_0,t)$
attains its maximum over $K$ at a unique point $t_0$.
Then $H$ is differentiable at $\theta_0$, and
\[
    H'(\theta_0)
    =
    \partial_\theta F(\theta_0,t_0).
\]
\end{lemma}

\begin{proof}
For every $\theta \in I$, the continuous function $t\mapsto F(\theta, t)$ attains its maximum $H(\theta)$ at some point of the compact metric space $K$.
For $h$ sufficiently small, choose $t_h\in K$ such that
\[
    H(\theta_0+h)=F(\theta_0+h,t_h).
\]
We first claim that $t_h\to t_0$ as $h\to0$. Indeed, if $h_n\to0$,
then by compactness any subsequence of $(t_{h_n})$ has a further
subsequence converging to some $t\in K$. Since
\[
    F(\theta_0+h_n,t_{h_n})
    \geq
    F(\theta_0+h_n,t_0),
\]
continuity of $F$ yields
\[
    F(\theta_0,t)\geq F(\theta_0,t_0).
\]
Thus $t=t_0$ by uniqueness of the maximizer, proving the claim.

For $h>0$, the maximality of $t_0$ at $\theta_0$ and of $t_h$ at
$\theta_0+h$ gives
\[
    \frac{F(\theta_0+h,t_0)-F(\theta_0,t_0)}{h}
    \leq
    \frac{H(\theta_0+h)-H(\theta_0)}{h}
    \leq
    \frac{F(\theta_0+h,t_h)-F(\theta_0,t_h)}{h}.
\]
The left-hand side converges to $\partial_\theta F(\theta_0,t_0)$.
By the mean value theorem, the right-hand side equals
$\partial_\theta F(\xi_h,t_h)$ for some $\xi_h$ between
$\theta_0$ and $\theta_0+h$, and hence has the same limit, since
$t_h\to t_0$ and $\partial_\theta F$ is continuous. The case $h<0$
is identical, with the inequalities reversed. Therefore
\[
    H'(\theta_0)=\partial_\theta F(\theta_0,t_0)
\]
as claimed.
\end{proof}

We need the following special case of Lemma~\ref{lem:danskin}.
\begin{lemma}[Danskin's theorem for the support function of a curve]\label{lem:danskin_special_case_cos_sin}
Let $x,y:[0,T]\to\mathbb R$ be continuous functions, and define
\[
    h(\theta)
    :=
    \max_{0\leq t\leq T}
    \bigl(x(t)\cos\theta+y(t)\sin\theta\bigr),
    \qquad \theta\in\mathbb R.
\]
Fix $\theta_0\in\mathbb R$ and suppose that the function $t\mapsto x(t)\cos\theta_0+y(t)\sin\theta_0$
attains its maximum on $[0,T]$ at a unique point $\tau_{\theta_0}$.
Then $h$ is differentiable at $\theta_0$ and
\[
    h'(\theta_0)
    =
    -x(\tau_{\theta_0})\sin\theta_0
    +
    y(\tau_{\theta_0})\cos\theta_0.
\]
\end{lemma}
\begin{proof}
Apply Lemma~\ref{lem:danskin} with $K=[0,T]$ and $F(\theta,t) :=x(t)\cos\theta+y(t)\sin\theta$.
Then $F$ is continuous and
$
    \partial_\theta F(\theta,t)
    =
    -x(t)\sin\theta+y(t)\cos\theta
$
is continuous on $\mathbb R\times[0,T]$. By assumption,
$t\mapsto F(\theta_0,t)$ has the unique maximizer
$\tau_{\theta_0}$. Hence
\[
    h'(\theta_0)
    =
    \partial_\theta F(\theta_0,\tau_{\theta_0})
    =
    -x(\tau_{\theta_0})\sin\theta_0
    +
    y(\tau_{\theta_0})\cos\theta_0,
\]
as claimed.
\end{proof}

\section*{Acknowledgements}
This work was supported by the DFG under Germany's Excellence Strategy
EXC 2044/2--390685587, Mathematics M\"unster:
\emph{Dynamics-Geometry-Structure}; by the DFG Priority Programme
SPP 2265, \emph{Random Geometric Systems}; and by the DFG Research
Training Group \emph{Rigorous Analysis of Complex Random Systems}
(RTG 3027, Project Number 524444762).

\section*{Declarations}

\subsection*{Statement on the use of generative AI}

This paper grew out of an extended dialogue between ChatGPT and the human
author. The project started with the question of determining the asymptotic
behavior of $\mathbb P(z\notin\mathcal K_T)$ as $T\to\infty$, now stated in
Proposition~\ref{prop:asympt_large_T}. A natural first approach was based on classical results on the winding
process of planar Brownian motion
\cite{Spitzer1958,Vakeroudis2011} and on explicit survival probabilities
for Brownian motion in planar wedges \cite{Goldman1996,ChupeauBenichouMajumdar2015}.  After substantial
manipulations of special functions, this approach, which is not presented
here, led to an explicit expression for $\mathbb P(z\in\mathcal K_T)$.
Considerable further work was required to reduce that expression to the
simple form now stated in
Theorem~\ref{theo:coverage_function_brownian_hull}.

Once this suggestive probabilistic form had been identified, ChatGPT proposed
a proof based on a Kac--Rice-type counting argument, later reformulated using
the one-dimensional change-of-variables formula,
Lemma~\ref{lem:one_dimensional_coarea_formula}. This approach essentially
led to Lemma~\ref{lem:coarea_radial_function}, but initially required a
rather involved direct evaluation of its right-hand side. Attempts to replace
these computations by a probabilistic argument led first to Remark~\ref{rem:another_set_same_radial}, then to the introduction of
the Gaussian ellipse and ultimately to the proof structure presented in this
paper.

ChatGPT played a crucial role at all stages of this process, including the
development and simplification of proofs and the exploration of related
consequences. All mathematical suggestions produced by ChatGPT were
critically examined and independently verified by the human author.

\subsection*{Conflict of interest statement}
The author declares that he has no conflict of interest.

\subsection*{Data availability statement}
We do not analyze or generate any datasets.

\bibliography{planar_brownian_convex_hull_bib}
\bibliographystyle{plainnat}

\end{document}